\documentclass[english,a4paper,12pt]{article}
\usepackage{amsfonts,amsmath,amsthm,graphicx,afterpage,authblk,setspace}

\newtheorem{theorem}{Theorem}
\newtheorem*{lemma}{Lemma}
\newtheorem{proposition}{Proposition}
\theoremstyle{definition}
\newtheorem{remark}{Remark}

\newcommand{\E}{\mathsf{E}}
\renewcommand{\P}{\mathsf{P}}
\newcommand{\Ii}{\mathcal{I}}
\newcommand{\I}{\mathrm{I}}

\title{Bounds for   expected maxima of   Gaussian processes and their
discrete approximations}
\author[1]{Konstantin Borovkov}
\author[2]{Yuliya Mishura}
\author[3,4]{Alexander Novikov}
\author[4]{Mikhail Zhitlukhin}
\affil[1]{School of Mathematics and Statistics, The University of Melbourne, Parkville 3010, Australia; e-mail: borovkov@unimelb.edu.au}
\affil[2]{Mechanics and Mathematics
Faculty, Taras Shevchenko National University of Kyiv, Volodymyrska str.~64, 01601 Kyiv, Ukraine; email: myus@univ.kiev.ua}%  +38-044-259-03-92
\affil[3]{School of Mathematical and Physical Sciences, University of Technology Sydney, PO Box 123, Broadway,
Sydney, NSW 2007,   Australia; email: Alex.Novikov@uts.edu.au}
\affil[4]{Steklov Institute of Mathematics, Gubkina str.~8,
119991, Moscow, Russia; email: mikhailzh@mi.ras.ru}
\date{}

\begin{document}
\maketitle

\begin{abstract}
The paper deals with the expected maxima of continuous Gaussian processes  $X = (X_t)_{t\ge 0}$ that  are H\"older continuous in $L_2$-norm and/or satisfy the opposite inequality for the $L_2$-norms of their increments.  Examples of such processes include the fractional Brownian motion  and some of its ``relatives" (of which several examples are given in the paper). We establish upper and lower bounds for $\E \max_{0\le t\le 1}X_t$ and investigate the rate of convergence to that quantity of its discrete approximation  $\E \max_{0\le i\le n}X_{i/n}$. Some further properties of these two maxima are established in the special case of the fractional Brownian motion.

\medskip



\end{abstract}

%\medskip \hfill {\em Dedicated to Bernt {\O}ksendal on occasion of his 70th birthday.}

\section{Introduction}

The problem this paper mostly deals with is finding upper and lower bounds for the expected maximum
\begin{equation}
\E\max_{0\le t\le \tau }X_{t}
%\quad\text{and}\quad \E\max_{0\le i \le n}X_{  i/n}.
\label{Emax}
\end{equation}
of a  zero-mean  continuous Gaussian process    $X = (X_t)_{t\ge 0}$ over a finite interval $[0,\tau]$. In what follows, we will always assume  that $\tau=1$ which does not restrict generality in the case of deterministic~$\tau$ (note that  the case of random $\tau$ will require application of
a more advanced technique, see e.g.~\cite{NoVa} and Section~1.10 in~\cite{mishura}). Computing the value of  expectation~\eqref{Emax} is an important question arising in a number of applied problems, such as finding the likely magnitude of the strongest earthquake to
occur this century in a given region  or the speed of the strongest wind gust
a tall building has to withstand during its lifetime etc. %  (see e.g.\ ~\cite{talagrand}).

We will mostly deal with processes $X$ satisfying one or both inequalities in the following condition: for some positive constants $C_i, H_i $, $i=1,2,$ one has
\begin{equation}
C_1 |t-s|^{H_1} \le  \|X_t - X_s\|_2\le  C_2|t-s|^{H_2}
\quad \text{for all}\quad t,s \ge 0,
\label{corr-ineq}
\end{equation}
where $\|Y\|_2=\sqrt{\E Y^2}$ denotes the $L_2$-norm of the random variable $Y.$

In the case when $H_1=H_2$ in (2), the process $X$ is a quasihelix in the respective $L_2$ space, in the terminology introduced in~\cite{kahane} (see also~\cite{kahane1}). So one can say that our paper mostly deals with the   expected maxima of generalized quasihelices, with substantial attention paid to the special case of a helix first introduced in~\cite{Kolm40}, which is the famous    fractional Brownian motion (fBm) process $B^H=(B_t^H)_{t\ge 0}$ with Hurst parameter $H\in (0,1]$. This is a continuous zero-mean Gaussian process with $B_0^H =
0$ and   covariance function
\begin{equation}
\E  B_{t}^{H} B_{s}^{H} =\frac{1}{2}%
(t^{2H}+s^{2H}-|t-s|^{2H}),  \quad t,s \ge 0,
\label{ffBm}
\end{equation}
in which case $\|B_t^H - B_s^H\|_2 = |t-s|^{H}$. The Hurst parameter  characterizes the nature  of dependency of the increments of the fBm.
For $H\in (0,\frac12)$ and $H\in (\frac12, 1]$, the increments of $B^{H}$ are  respectively  negatively and positively correlated, whereas in the special case $H=\frac12$ the process $B^{1/2}$ is the standard Brownian motion which  has independent increments.  When $H=1$, the trajectories of the fBm process are a.s.\ rays with a random slope: $B_t^1 = \xi t$, $t\ge 0,$ where
$\xi$ is a standard normal random variable. (It is instructive to note that, at the left end of the spectrum of $H$ values, as $H\to 0$, the limiting in the sense of convergence of finite-dimensional distributions for $B^H$ process   will be just a Gaussian white noise plus a common random variable, see Section~\ref{Sect_add_prop}.)   Other examples of processes satisfying
\eqref{corr-ineq} are given in Section~\ref{examples}.

% In the case $H=1$, the path of the process are a.s. straight lines with
% random slope, $B_t^H = \xi t$, where $\xi$ is a standard normal random
% variable.

It is hardly possible to obtain a closed-form analytic expression for the expectation
\eqref{Emax}, even in the classical case of the fBm. For the standard Brownian motion $B^{1/2}$, the exact value
of the expected maximum is $\sqrt{\pi/2}$, while for all other $H$ within $(0,1)$ no closed-form
expressions for the expectation  are  known.

\begin{figure}[ht]
\centering
\includegraphics[width=0.5 \linewidth]{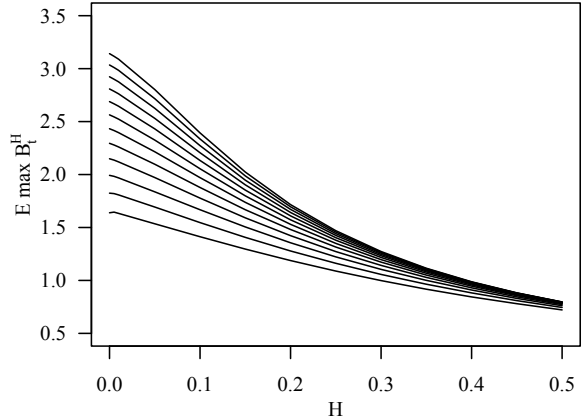}%pdf}%
\caption{Monte Carlo estimates for the values of $\E \max_{1\le i\le n} B^H_{i/n},$   $n=2^5,\ldots,2^{16}$. The $k$th lowest curve shows results for $n=2^{k+4},$ $k=1,\ldots, 12.$ }
\label{fig1}
\end{figure}

One can try to evaluate expectation~\eqref{Emax}  numerically, using Monte Carlo
simulations and the ``discrete approximation''
\begin{equation}\label{Emax_n}
\E \max_{0\le t\le 1}X_{t} \approx \E \max_{0\le i \le
n}X_{ i/n},
\end{equation}
choosing a large enough~$n$. One difficult question  that arises in doing so is what value of $n$ should be chosen to
get a sufficiently good approximation. Fig.~\ref{fig1} presents the results of our
attempts to estimate the value of~\eqref{Emax} using approximation~\eqref{Emax_n} for fBm's with different Hurst parameter values. The twelve curves show the  results for 12 different values $n=2^5,2^6,\ldots,2^{16}$ (the values $n=2^k$ were chosen since  paths of the fBm's were generated using the Davies--Harte~\cite{Davies} algorithm which employs the fast discrete Fourier transform, and the latter is convenient to compute when the numbers of points  is a power of two). For each combination of the values of $H$ and $n,$ we simulated  $5\times 10^5$ paths and the same number of antithetic paths. In all cases, the 99.9\% confidence intervals were of lengths less than~0.02.

As seen from Fig.~\ref{fig1}, approximation~\eqref{Emax_n} does not seem to work well for  the values of $H$ close to zero, as one can hardly notice in the respective part of the plot any convergence of the estimates  to a particular value  as~$n$ grows.

%,  and even $n=2^{16}$ points  does not seem to give a reliable approximation.

This leads to the natural problem of bounding   the difference
\begin{equation}\label{Emax_diff}
\Delta_n (X):= \E \max_{0\le t\le 1}X_{t}- \E \max_{0\le i \le n}X_{ i/n}\ge 0
\end{equation}
for a Gaussian process $X$ satisfying \eqref{corr-ineq} and, in particular, for the
fBm. For the latter, in view of the observed slow convergence, we are
particularly interested in the case of  small~$H$ values.

The main results of the paper show that, for a process $X$
satisfying~\eqref{corr-ineq},
\begin{equation}\label{Results}
\frac{K_1C_1}{\sqrt{H_1}} \le \E\max_{0\le t\le 1}X_{t} \le \frac{K_2C_2}{\sqrt H_2}, \qquad
\Delta_n (X) \le \frac{K_3C_2\sqrt{\ln n}}{n^{H_2}}
\end{equation}
with some absolute constants $K_i\in (0,\infty),$ $i=1,2,3$.  For a family of processes $X^{H_1},$ $H_1\in (0,1),$ satisfying the left inequality in~\eqref{corr-ineq} with $C_1=C_1^{H_1}>c_0>0$ (or even $C_1^{H_1}\gg \sqrt{H_1}$), the first relation in~\eqref{Results} shows that $\E\max_{0\le t\le 1}X_{t}^{H_1}\to \infty$  as $H_1\to 0$, while the last one suggests a poor convergence rate for the discrete
approximation.

Furthermore, we provide some lower bounds for the discrete approximation rate in Section~\ref{Section_discr}.
In addition, using these bounds, we study the properties of the expected maximum of the fBm  and its discrete approximation. We show
that the latter is continuous as a function of $H$ and find its limit as $H\to 0$.

Computing bounds for the extrema of Gaussian processes is a large research area, see e.g.\ monographs
\cite{talagrand,piterbarg} and references therein. A powerful general method for obtaining
bounds for the expected maximum of a stochastic process is based on a generic
chaining technique, for which \cite{Talagrand1,talagrand} are exhaustive sources.
However, for the particular problem we are dealing with here, the general method seems
to be   unnecessary cumbersome (moreover, it includes a non-trivial step of choosing
partitions of the time interval~\cite[Section~2.3]{talagrand}). The simple argument we provide below already allows one to obtain lower and   upper
bounds of the same order of magnitude~$1/\sqrt{H}$.

There are quite a few results in the literature concerning a related problem on   evaluating the so-called  Pickands' constant  $P_H$ that  plays
a fundamental role in the theory of extrema of Gaussian processes.
The constant is defined as
\[
P_{H}:=\lim_{\tau \rightarrow \infty }\frac{1}{\tau}\E\max_{0\le t\le \tau}\exp \bigl(\sqrt{2}B_{t}^{H}-t^{2H}\bigr).
\]
Recent results on  theoretical and numerical estimates for $P_H$ include \cite{DiekerYakir,DebKis,Harper,Shao}.
In particular,   \cite{Shao} contains an auxiliary result giving an upper
bound for the expected maximum of the fBm.

With regard to the problem on bounding $\Delta_n (X)$, note that
there are many results about approximating the  paths of continuous time processes by those of
discrete time ones, that can yield bounds for their functionals, including maxima
and minima. However,  the path-wise approximation may  not be necessary in the
problem we consider, as we are only interested in the {\em expected\/} maximum.
Regarding bounds for the expected maximum, we can only mention the preprint~\cite{PiterbargIvanov}, where the maximum of the
geometric fBm  was considered in the context of option pricing in the fractional
Black--Scholes model. The bounds obtained in~\cite{PiterbargIvanov}   for the discrete approximation are of the
order $O(n^{-H}\sqrt{\ln n})$, but the constant in the bound depends on~$H$.

The paper is organized as follows. In Section~\ref{Section_Bounds} we prove   lower and upper
bounds for the expected maximum. In Section~\ref{Section_discr} we establish an upper bound for $\Delta_n (X)$
and  also consider the problem of bounding the difference from below. In
Section~\ref{Sect_add_prop}, we establish properties of the expected maximum of the fBm $B^H$ and its approximation as   functions of~$H$ and~$n$, in particular, their continuity and
limits as~$H\to 0$. Section~\ref{examples} contains further examples of  Gaussian processes satisfying~\eqref{corr-ineq}, so that the results of our Theorems~\ref{theo1}--\ref{theo3} are applicable to them as well. The Appendix contains some known inequalities for Gaussian processes that are used in the paper.

\section{Bounds for the expected maximum}
\label{Section_Bounds}

In this section we use a combination of Sudakov's inequality and
Fernique-Talagrand's generic chaining bound to give  short derivations of upper and lower bounds for   the expected maximum of a centered Gaussian
process satisfying~\eqref{corr-ineq}. In particular, the bounds imply that, for a family of processes $X^H,$ $H\in (0,1)$,  such that $X^H$ satisfies~\eqref{corr-ineq} with $H_1=H_2=H$, the constants $C_i$ being independent of~$H$,    the expected maximum of $X^H$ on $[0,1]$ is of the order of magnitude  of
$1/\sqrt H$ as $H\to 0,$

\begin{theorem}\label{theo1}
{\rm (i)} If there exist  $C>0$ and $H\in(0,1)$ such that $\|X_t-X_s\|_2 \ge C
|t-s|^{H}$ for any $t,s\in [0,1]$ then
\[
\E\max_{0\le t\le 1}X_{t}
\ge \frac{C}{\sqrt{4H\pi e \ln 2}}  > \frac
 C{5\sqrt H}.
\]

{\rm (ii)}  If there exist  $C>0$ and $H\in(0,1)$ such that $\|X_t-X_s\|_2 \le C
|t-s|^{H}$ for any $t,s\in [0,1]$ then
\begin{equation}
\label{upper_b}
\E\max_{0\le t\le 1} X_{t} \le LC\sqrt{\frac{2\pi }{H\ln^3 2}}\, {\rm erfc}\, \sqrt{\frac{\ln 2}2 H}
%\frac{LC\sqrt{2\pi}}{\sqrt{H(\log 2)^3}}
<
\frac{16.3C}{\sqrt H},
\end{equation}
where $L< 3.75$ is the constant from the generic chaining bound {\em (Proposition~\ref{tala} in
the Appendix)} and ${\rm erfc}$ is the complementary error function.
% \[
% \E \sup_{t\in \lbrack 0,1]}X_{t} \le \frac{C}{\sqrt{H}}
% (16\sqrt\pi + 20\log\frac32\cdot\sqrt{2H^3\log 3}) \le
% \frac{C}{\sqrt{H}} (29 + 13H^{3/2}).
% \]
%\end{itemize}
\end{theorem}

\begin{remark}
Note that the basic assumption of continuity of $X$ is not needed in part~(i). In the general case, the bound will hold for $\E\sup_{0\le t\le 1} X_t$ which can be defined as, say, in (0.1) in~\cite{Talagrand1}. On the other hand, that   assumption is superfluous in part~(ii), where the continuity of $X$ will follow from the assumed bound for $\|X_t-X_s\|_2$, the Kolmogorov--Chentsov theorem and the usual scaling argument for Gaussian random variables.
\end{remark}

\begin{remark}
The lower bound from Theorem~\ref{theo1} is noticeably more accurate than the upper one.
Fig.~\ref{fig2} presents the graph of the lower bound and the value of the
expected maximum for fBm  computed numerically by the
approximation at $n=2^{16}$ points. The upper bound is far above these
plots. A better upper bound for $H < 0.5$ was obtained in~\cite[Lemma
6]{Shao} using Borel's inequality, but   simulations show that it is still considerably
greater than the true value (\cite{Shao} provides an additional term under the
square root, but the  proof there is  more technical than the simple
argument given below that nevertheless gives the same order of magnitude~$1/\sqrt{H}$). For $H  \ge 1/2$, a much better bound than~\eqref{upper_b} immediately follows from Sudakov-Fernique's inequality (Proposition~\ref{sudakov lemma}):
\[
\E \max_{0\le t\le 1} X_t \le C \E \max_{0\le t\le 1} B_t^{1/2} = C \sqrt{2\pi}.
\]
\end{remark}

\begin{figure}[ht]
\centering
\includegraphics[width=0.5 \linewidth]{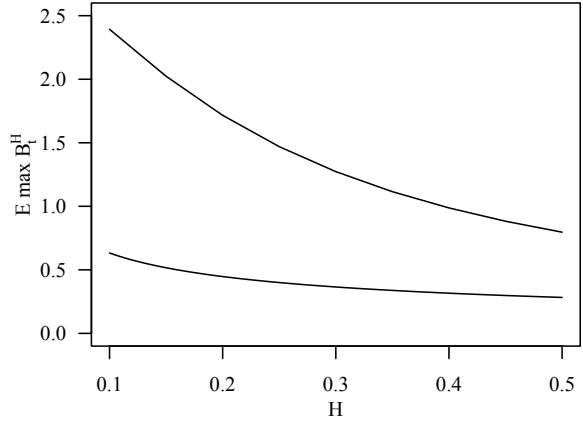}%pdf}%
\caption{The lower bound $1/(5\sqrt{H})$ for the expected maximum of the fBm's for different values of $H\in (0,1)$ and the Monte Carlo estimates for that expectation using approximation~\eqref{Emax_n} with $n=2^{16}.$ }
\label{fig2}
\end{figure}

\begin{proof}
Without loss of generality one can assume that $C=1$ both in~(i)  and~(ii), for
one can always switch to considering the process $X_t/C$.

(i)~Let $n$ be the integer part of  the value $x:=e^{1/(2H)}$ that maximizes  the function $x^{-H}\sqrt{\log_{2}x}$. Then, in view
of Sudakov's inequality (Proposition~\ref{sudakov lemma}),
\begin{align}
\E\max_{0\le t\le 1}X_t & \ge  \E\max_{0\le i\le n}X_{i/n}\ge \sqrt{\frac{\log_{2}(n+1)}{n^{2H}2\pi }}
\notag
 \\
 & \ge
\sqrt{\frac{\log_{2}x}{x^{2H}2\pi }} = \frac{1}{2 \sqrt{ H\pi e \ln 2}}
=\frac{0.2055\ldots}{\sqrt{H}}.
\label{lower_b}
\end{align}

(ii)  Consider the sets
$T_0:=\{1/2\}$, $T_{n}:=\{j 2^{-2^{n}}, j=1,\ldots,2^{2^n}\}$ for $n\ge
1$, so that $|T_n| =2^{2^{n}}$. Using the generic chaining bound (Proposition~\ref{tala} in
the Appendix), we have, setting $a:=\frac12 \ln 2$ and changing variables $u:=e^{ax}$, that
\begin{align*}
\E \max_{0\le t\le 1} X_t
 &
  \le L \sum_{n=0}^\infty 2^{n/2} 2^{-H2^n}
  \le L \int_0^\infty   2^{x/2- H 2^{x-1}}  dx
     \\
     &
     = L \int_0^\infty \exp\{  a(x - H e^{2ax})\}  dx
   \le \frac{L}a \int_1^\infty e^{-aHu^2}du
   \\
    & =
   L\sqrt{\frac{2\pi }{H\ln^3 2}}\, {\rm erfc}\, \sqrt{\frac{\ln 2}2 H}
    <
L\sqrt{\frac{2\pi }{H\ln^3 2}}.
\end{align*} \end{proof}

\section{Bounds for the discrete approximation}
\label{Section_discr}

In this section we study the discrete approximation to the expected maximum of~$X$. The main result is Theorem~\ref{theo2} giving an upper bound for the approximation error~\eqref{Emax_diff}. A rather crude lower bound for the error is obtained in Theorem~\ref{theo3}.

\begin{theorem}\label{theo2}
Assume  that  $\|X_t-X_s\|_2 \le C|t-s|^{H}$
for any $t,s\in[0,1]$ with some constants $C>0$ and $H\in(0,1)$. Then, for any $n\ge 2 ^{1/H},$
\begin{equation*}
\Delta_n (X) \le \frac{2C\sqrt{\ln n}}{n^{H}}\biggl(1+\frac4{n^H}+\frac{0.0074}{(\ln n)^{3/2}}\biggr).
\end{equation*}%
%
% where $C_3$ is a constant which does not depend on $H$ and $n$ and can be
% chosen
% \[
% C_3 \le \frac{4}{\log 2} + \sqrt{\frac{2}{\pi}} \le 7.
% \]
\end{theorem}

\begin{proof}
As before, we will assume without loss of generality that $C=1$. Suppose that $n\ge  2 ^{1/H}$ is fixed  and let %$n_{k}:=n^{k},$
\[
\epsilon _{k} := %\frac{2\sqrt{\ln n^{k}}}{n^{Hk}} =
 2  {n^{-Hk}}\sqrt{k \ln n}, \quad k\ge 1.
\]
From the continuity of $X$ and monotone convergence theorem it follows  that
\[
\begin{split}
\Delta_n (X)
  &=
\sum_{k=1}^\infty \E \Bigl( \,\max_{0\le i\le n^{k+1}}X_{  i/ n^{k+1}}-\max_{0\le i \le
n^k}X_{i/ n^{k}}\Bigr)
 \\ &=
\sum_{k=1}^\infty \int_{0}^{\infty }\P\Bigl(\,\max_{0\le i\le n^{k+1}}X_{  i/ n^{k+1}}-\max_{0\le i \le
n^k}X_{i/ n^{k}}\ge u\Bigr) du
\\ & =
 \sum_{k=1}^\infty   \biggl(\int_{0}^{\epsilon_k } \cdots + \int_{\epsilon_k }^\infty \cdots\biggr) =:  \sum_{k=1}^\infty (I_k +J_k).
\end{split}
\]
Clearly, $I_k \le \epsilon_k$. To estimate $J_k,$ introduce the sets
$T_m: = \{j/m, j=0,1,\ldots, m\}$, $U_k(s):=\bigl\{t\in T_{n^{k+1}} : t\neq s, -\frac{1}{2n^k}< t-s \le \frac{1}{2n^k} \bigr\}$ and note  that
\begin{align*}
\P\Bigl(\,\max_{0\le i\le n^{k+1}}X_{  i/ n^{k+1}}&-\max_{0\le i \le
n^k}X_{i/ n^{k}}\ge u\Bigr)
   \le
  \sum_{s\in T_{n^k}}\sum_{t\in U_{k}(s)}\P%
(X_{t}-X_{s}\ge u)
 \\
 & \le
 \frac{1}{\sqrt{2\pi }}\sum_{s\in T_{n^{k}}}\sum_{t\in U_k(s)}\frac{1}{u(2n^{k})^{H}}\exp \biggl(
-\frac{u^{2}}{2(1/(2n^{k}))^{2H}}\biggr)
 \\
 & \le
\frac{n\cdot (n^{k})^{1-H}}{2^{H}\sqrt{2\pi }  u}\exp \biggl( -\frac12  u^{2} (2n^{k})^{2H}\biggr) ,
\end{align*}
where the second line follows from the well-known bound for the Mills' ratio  of the normal distribution:
% \begin{equation*}
%\P (X_{t}-X_{s}\ge u)\le \frac{1}{\sqrt{2\pi }}\cdot \frac{%
%|t-s|^{H}}{u}\exp \biggl( -\frac{u^{2}}{2|s-t|^{2H}}\biggr),
%\end{equation*}%
%which is the well-known inequality $
\[
\P(\xi \ge x) \le
\frac\sigma{x  \sqrt{2\pi}} \exp \biggl(-\frac{x^2}{2\sigma^2}\biggr), \quad x >0,
\]
for  $\xi\sim N (0,\sigma^2)$, and the observation that $\mbox{Var}\, (X_t - X_s) \le (2n^k)^{-2H}$ for $t\in U_k (s),$ $s\in T_{n^k}$. The inequality in the third line holds since   $|T_{n^{k+1}}|=n^{k+1}+1=n\cdot n^{k}+1$ (note that not all the points $t\in T_{n^{k+1}}$ contribute to the sum).

Now applying the bound $\int_a^\infty u^{-1}e^{- {u^2}/{2}}du\le a^{-2}e^{- {a^2}/{2}}$, $a>0$,
we get
\begin{align*}
J_k
&
 \le
 \frac{n\cdot (n^{k})^{1-3H}}{2^{3H}\sqrt{2\pi }   \epsilon _{k}^{2}}
 \exp \biggl(-\frac12  \epsilon_{k}^{2}(2n^{k})^{2H} \biggr)
 \\
 & = \frac{n\cdot n^{k(1-H)}}{2^{3H+2}\sqrt{2\pi }   k\ln n}
 \exp \bigl(-2^{2H+1} \ln n^k \bigr)
  = \frac{n}{2^{3H+2}\sqrt{2\pi }    \ln n}\cdot \frac{n^{k(1-H-2^{2H+1})}}{k}.
\end{align*}
Setting $z:= n^{ 1-H-2^{2H+1} },$ we conclude that
\begin{equation}
\label{Delta}
\Delta_n (X)
\le \sum_{k=1}^\infty\biggl( \epsilon_k
+
 \frac{n}{2^{3H+2}\sqrt{2\pi }    \ln n}\cdot \frac{z^k}{k} \biggr) .
\end{equation}
Bound the contribution to the sum on the right-hand side from the first summands in the terms:
\begin{equation}
\label{Sum_ep}
\sum_{k=1}^{\infty }\epsilon_{k}
 =
 2\sqrt{\ln n} \sum_{k=1}^{\infty }\frac{ k^{1/2}}{ n^{Hk}}
 =
  \frac{ 2\sqrt{\ln n}}{n^H} g(n^{ H}),
\end{equation}
where $g(x) := x\mbox{Li\,}_{-1/2}(1/x),$ $x>1,$ $\mbox{Li\,}_k(z)$ being the polylogarithm function. It is not hard to show that $g_1(x):=x(g(x)-1),$ $x>1,$ is a decreasing function, $g_1(2)<4$, so that $g(x) <1 + 4/x$ for $x\ge 2$.  Therefore, the right-hand side of \eqref{Sum_ep} is less than $2 n^{-H} \sqrt{\ln n} (1 + 4n^{-H}),$ $n\ge 2^{1/H}$.

Next note that  since $z=(n^H)^{(1-H-2^{2H+1})/H}<0.008$ when $n\ge 2^{1/H}$ and    $g_2(x):=-\ln (1-x)/x,$ $x>0,$ is an increasing function, one has $\sum_{k=1}^{\infty }z^k/k = -\ln (1-z) <g_2(0.008) z<1.005 z$.
Hence the contribution to the sum from the second  summands in the terms on the right-hand side of \eqref{Delta} does not exceed
\begin{align*}
\frac{n}{2^{3H+2}\sqrt{2\pi }\ln n}  \cdot 1.005 n^{1-H-2^{2H+1}}
  &= \frac{1.005 n^{2(1-2^{2H})}}{2^{3H+3}\sqrt{2\pi }(\ln n)^{3/2}}  \cdot \frac{2\sqrt{\ln n}}{n^H}
  \\
  &
   <
   \frac{0.0074}{(\ln n)^{3/2}} \cdot \frac{2\sqrt{\ln n}}{n^H},
\end{align*}
where we used the inequality $n^{2(1-2^{2H})} = (n^H)^{2(1-2^{2H})/H}\le 2^{- 2\ln 4}$ which  holds for $n^H\ge 2$ since $\sup_{H\in (0,1)} 2(1-2^{2H})/H= - 2\ln 4$.

The above bound, together with inequality~\eqref{Delta} and our bound for the right-hand side of \eqref{Sum_ep}, completes the proof of Theorem~\ref{theo2}.
\end{proof}

\begin{remark} In the special case when $X=B^H$ is an fBm, it is possible to obtain a similar bound  in a simpler way, but the result will contain a constant depending on~$H$: for any $H\in (0,1)$ and $n\ge 1$,
\[
\Delta_n (B^H) %\E \max_{t\le1}B^H_{t}-\E\max_{i\le n}B^H_{\frac in}
 \le {n^{-H} } {\sqrt{a(H)+ \ln n}},
\]
where $a(H)$ is a function of $H$ such that $a(H) \to
\infty$ as $H\to 0$ (see~\eqref{aH} below).

Let $B^{H,n} $ be a process whose trajectories are random polygons with nodes at the points $(i/n, B^H_{ i/n}),$ $i=0,1,\ldots, n.$  Then
\[
\Delta_n (B^H)
\le  \E   \max_{0\le t\le 1} (B_t^H - B^{H,n}_t) = \E Y,
\]
where $Y:=\max_{1\le i\le n} Z_i^{H,n}$,
\[
Z_i^{H,n} := \max_{ (i-1)/n\le t \le i/n} (B^H_t -  B_t^{H,n}), \quad  i=1,2,\ldots, n,
\]
being identically distributed random variables such that
\[
Z^{H,n}_1 \stackrel{d}{=} n^{-H} \max_{0\le t\le1}(B_t^H - t B^H_1),
\]
due to the self-similarity of the fBm.

Therefore, for any $\lambda >0,$ by Jensen's inequality one has
\[
(\E Y)^2 \le   \frac 1\lambda \ln \E e^{\lambda Y^2}
  =
   \frac 1\lambda \ln \E \max_{1\le i \le n} e^{\lambda(Z_i^{H,n})^2 }
  \le
 \frac 1\lambda\ln \bigl(n \E e^{\lambda(Z_1^{H,n})^2}\bigr).
\]
Now taking $\lambda:=n^{2H}$ we obtain the desired bound for $\Delta_n (B^H) $ with
\begin{equation}
\label{aH}
a(H): = \ln \E \exp\biggl[\Bigl(\max_{0\le t \le1}(B_t^H - t B^H_1)\Bigr)^2\biggr].
\end{equation}
Again using Jensen's inequality, we see that
\begin{align*}
a(H) &>  \E\Bigl(\max_{0\le t\le 1} (B_t^H - t B_1^H)\Bigr)^2
  > \E    \Bigl( \max_{0\le t\le 1} B_t^H -  \max\{0,B_1^H\}\Bigr)^2
  \\
  & \ge \E \Bigl( \max_{0\le t\le 1} B_t^H \Bigr)^2 -\E \Bigl( \max\{0,B_1^H\}\Bigr)^2
  \ge  \Bigl(\E \max_{0\le t\le 1} B_t^H \Bigr)^2 -\E \Bigl( \max\{0,B_1^1\}\Bigr)^2,
\end{align*}
which tends to infinity as $H\to 0$ in view of Theorem~\ref{theo1}(i).
\end{remark}

\begin{remark}
Note also that, in the even more special case of the standard Brownian motion ($H=1/2$),
the exact asymptotics of  $\Delta_{n}(B^{1/2})$ are well-known and contain no logarithmic factor:
\[
\Delta_{n}(B^{1/2})= (\alpha +o(1))n^{-1/2},\quad  n\to \infty ,
\]%
where $\alpha =-\zeta(1/2)/\sqrt{2\pi }=0.5826\ldots$ and  $\zeta(\cdot) $ is the  Riemann zeta function (see~\cite{Chernoff, Siegmund}).
One may expect that the logarithmic factor  in the general bound from Theorem~\ref{theo1} is also superfluous, but that the power factor $n^{-H}$ gives the correct order of magnitude for $\Delta_n(B^H)$.
\end{remark}

Unfortunately, in the general case we are not aware of any good lower bound for the discrete approximation rate. In the case of our particular interest, when~$H\to
0$, our next theorem shows that the number of points $n=n(H)$ should grow as an  exponential of $1/H$ to provide a
reasonable approximation. Namely, we will show that  $n(H)$ should be at least of the order of~$c^{1/H}$ for some constant
$c>1$ or, which is the same,  that $ (n(H))^H$
should be bounded away from below from one as $H\to 0$.

\begin{theorem}\label{theo3}
Let $X^H = (X_t^H)_{t\le1}$, $H\in(0,1)$, be a family of zero-mean Gaussian processes such
that, for some constants $C_1,C_2>0$ and all $H\in(0,1),$ one has
\[
C_1 |t-s|^{H} \le \|X_t^H-X_s^H\|_2 \le  C_2 |t-s|^{H}
\quad \mbox{for all}\quad  t,s\in[0,1].
\]
Suppose $n(H)\ge 2$ is such that
\begin{equation}
\label{7}
\limsup_{H\to 0} \Delta_{n(H)} (X^H) < \infty.
\end{equation}
Then $\liminf_{H\to 0} (n(H))^H> 1.$
\end{theorem}

\begin{proof}
Suppose the contrary: relation~\eqref{7} holds, but there is a sequence $H_m\to 0$  such that
$ H_m \ln n (H_m) \to 0$.

Set $n_m:=n(H_m)$ and choose $k_m$   such that $ {2^{k_m-1}} < \log_2 n (H_m) \le  2^{k_m} $. Note that, in view of the above assumption, one has
\begin{equation}
\label{2mH}
2^{k_m}=o (H_m^{-1 }).
\end{equation}

Now introduce the sets  $T_0:=\{1/2\}$,
$T_j:=\{i/2^{2^j},\, i=1,\ldots,2^{2^j}\}$ for $1\le j\le k_m,$ and $T_{k_{m+1}} := T_{k_m} \cup
\{i/n_m: i=0, 1,\ldots, n_m\}$. Note that $|T_j| \le 2^{2^j}$ for all $j \le k_{m+1}$. Hence we can apply the generic chaining bound (Proposition~\ref{tala}) with $T=T_{k_{m+1}}$ to get
\begin{align*}
\E \max_{0\le i\le n_m} X_{ i/{n_m}}^{H_m}
 &\le
  L \sum_{j=0}^{k_m} 2^{  j/2}C_2 2^{-H_m2^j}
  + L 2^{ (k_m+1)/2}
 C_2 2^{- H_m 2^{k_m}}
 \\
 &\le
 LC_2 \int_0^{k_m}2^{ x/2}  2^{-H_m 2^x/2}dx + o(H^{-1/2}_m)
\end{align*}
in view of~\eqref{2mH}. Setting $a:=\frac12 \ln 2$ and making the change of variables $u:=e^{ax}$, we see that the integral on the right-hand side of the above formula equals
\[
\int_0^{k_m} e^{a x }  \exp \bigl(-a H_m e^{2ax}\bigr) dx
 = \frac1{a}\int_1^{2^{k_m/2}} e^{-aH_mu^2} du\le \frac{2^{k_m/2}}{a} = o(H^{-1/2}_m)
\]
again in view of~\eqref{2mH}. Thus, $\E \max_{0\le i\le n_m} X_{ i/{n_m}}^{H_m}  = o(H^{-1/2}_m).$

From here and Theorem~\ref{theo1}(i), one has
\[
\Delta_{n_m} (X^{H_m})=
\E  \max_{0\le t\le 1} X_t^{H_m}
 - \E \max_{0\le i\le n_m} X_{ i/{n_m}}^{H_m}
  \ge \frac{C_1}{5   }  H_m^{-1/2}+o(H^{-1/2}_m) \to \infty
\]
as $m\to\infty$, which contradicts \eqref{7}. Theorem~\ref{theo3} is proved.
\end{proof}

\section{Further properties of the expected maximum of fBm }
\label{Sect_add_prop}

Consider the functions
\[
f(H): = \E \max_{0\le t\le1} B^H_t, \quad f(H,n) := \E \max_{0\le i\le n} B^H_{i/n}, \quad H\in (0,1).
\]
It follows from Sudakov--Fernique's inequality (Proposition~\ref{lemmSF}) that both $f(H)$ and $f(H,n)$ are non-increasing in $H$.  Theorem~\ref{theo1}(i) implies that $f(H) \to \infty$ as $H\to 0$.

The main goals of this section are  to show that $f(H)$ and $f(H,n)$ are
continuous in $H$ and find the limit of $f(H,n)$ as $H\to 0$.

We will start with an auxiliary lemma about the limit of finite dimensional
distributions of $B^H$, which is of independent interest. Let $\xi=(\xi_{t})_{ t\ge 0}$ be the standard Gaussian white noise,
so that  $\xi_t$ are i.i.d.\ standard normal   random variables. Define the process $Z=(Z_t)_{t\ge 0}$ by setting  $Z_t :=
(\xi_t-\xi_0)/\sqrt 2$, $t\ge 0.$

\begin{lemma}
As $H\to 0,$  $B^H\stackrel{d}{\longrightarrow} Z$   in the sense of convergence of  finite-dimensional distributions.
\end{lemma}

\begin{proof}
Since all the finite-dimensional distributions of $B^H$ and $Z$ are zero-mean Gaussian,
it suffices to establish convergence of the covariance function of $B^H$ to
the one of $Z$. The latter
is clearly  given by
\begin{equation*}
\E Z_s Z_t =  \left\{
\begin{array}{ll}
0   &\text{if $t=0$ or $s=0$}, \\
\frac12 &\text{if $t\neq s$ and $t,s>0$}, \\
1  &\text{if $t=s>0$}.
\end{array}
\right.
\end{equation*}
That the covariance function~\eqref{ffBm} of $B^H$ converges to the same
expression as $H\to 0$ is obvious.
\end{proof}

\begin{theorem}
{\rm (i)} For any fixed $n\ge 1,$
\[
\lim_{H\to 0} f(H,n) = \frac{1}{\sqrt 2} \E \Bigl(\max_{1\le i\le n} \xi_i\Bigr)^+,
\]
where $\xi_i$ are  i.i.d.\ standard normal random variables, $x^+:=\max\{0,x\}.$

{\rm (ii)}  For any fixed $n\ge 1,$ the function $f(H,n)$ is continuous in $H$ on $[0,1).$ The function $f(H)$ is continuous on $(0,1)$.

{\rm (iii)}  Let $H_k\in(0,1)$ and $n_k\ge 1$ be two sequences,   $H_k\to 0$ as $k\to\infty.$ Then the sequence $f_k := f(H_k,n_k)$ is bounded if and only if~$n_k$ is bounded.
\end{theorem}

Note that part~(iii) could easily be  proved  if we knew that the function $f(n,H)$ is monotone in $n$. However, we do not have a simple proof of this fact, and therefore will present a
proof based on the bounds obtained in the previous sections.

\begin{proof}
Assertion~{\rm (i)}  follows from the lemma, with the positive part of the maximum appearing  because~$Z_0 = 0$.

\medskip

{\rm (ii)} Suppose $0\le H_1<H_2<1$. As we noted above, $f(H_1,n) \ge f(H_2,n)$
for any $n\ge 1$, and applying Chatterjee's inequality (Proposition~\ref{chat}) we get
\[
0\le f(H_1,n) - f(H_2,n) \le  \sqrt{\alpha_n(H_1,H_2) \ln n},
\]
where, for $H_1>0,$
\begin{align}
\alpha_n(H_1,H_2)
  &=\max_{1\le j\le n}(({j}/{n})^{2H_1}-(j/n)^{2H_2})\le
\max_{0\le x\le 1}(x^{2H_1}-x^{2H_2})
 \notag
 \\
  &= (H_1/H_2)^{\frac{H_1}{H_2-H_1}} - ({H_1}/{H_2})^{\frac{H_2}{H_2-H_1}}
  \le \frac{H_2 - H_1}{eH_1},
  \label{bound_al1}
\end{align}
while, for $H_1 = 0,$
\begin{align}
\alpha_n(0, H_2) = \max_{1\le j\le n}(1-(j/n)^{2H_2})
= 1- n^{-2H_2}.
  \label{bound_al2}
\end{align}
Now the continuity of $f(H,n)$ at any $H\in (0,1)$ and at $H=0$ follows respectively from~\eqref{bound_al1} and~\eqref{bound_al2}.

The continuity of $f(H)$ at any
$H\in(0,1)$ is obvious from the continuity of $f(H,n)$ and the upper bound for $\Delta_n (B^H)\equiv f(H,n) - f(H)$ from Theorem~\ref{theo2}.

\medskip

{\rm (iii)} If $n_k$ is bounded, the boundedness of $f_k$ follows from (i) and (ii).
In the case of unbounded  sequence~$n_k$, it will   suffice to prove that   $ f_k  \to \infty$ provided that   $H_k \to 0$ and $n_k \to \infty$ as $ k\to\infty.$

It follows from Theorem~\ref{theo2} that, for $n\ge 2^{1/H},$ one has
\[
f(H) - f(H,n) \le \frac{6\sqrt{\ln n}}{n^{H}}
 +\frac{2\cdot 0.0074}{n^H \ln 2^{1/H}}
 <  \frac{6\sqrt{\ln n}}{n^{H}} + c_1, \quad c_1:= 0.0107.
\]
Hence we obtain from  Theorem~\ref{theo1}(i) that
\[
f(H,n) \ge \frac1{5\sqrt{H}} - \frac{6\sqrt{\ln n}}{n^{H}} - c_1, \quad n\ge 2^{1/H}.
\]
On the other hand, in view
of Sudakov's inequality (Proposition~\ref{sudakov lemma}),
for any $n\ge 1$ one has (cf.~\eqref{lower_b})
\[
f(H,n) \ge \frac{c_2\sqrt{\ln n}}{n^{H}}, \quad c_2:=(2\pi \ln 2) ^{-1/2}.
\]
Therefore,
\begin{align}
f(H,n) & \ge \max\biggl\{ \frac1{5\sqrt{H}} - \frac{6\sqrt{\ln n}}{n^{H}}, \frac{c_2\sqrt{\ln n}}{n^{H}} \biggr\}  - c_1
\notag
\\
 & \ge  \frac{c_2  }{(6+c_2)5\sqrt{H}}- c_1, \quad n\ge 2^{1/H},
 \label{fHn1}
\end{align}
and
\begin{equation}
f(H,n) \ge \frac{c_2}2 \sqrt{\ln n},   \quad n<  2^{1/H}.
 \label{fHn2}
\end{equation}
Now the desired claim immediately follows from~\eqref{fHn1} and~\eqref{fHn2}.
\end{proof}

\section{Examples}
\label{examples}
As we mentioned in the Introduction, a fundamental  example of a zero-mean Gaussian
process satisfying~\eqref{corr-ineq} is the fBm~$B^H$ (in that case, \eqref{corr-ineq} holds with $H_1=H_2=H,$ $C_1=C_2=1$). In this section we give some further examples of such processes.

\smallskip

\noindent{\bf 1.~Sub-fractional Brownian motion.} This is a centered Gaussian process $C^{H}=(C_t^H)_{t\ge 0}$ with parameter $H\in (0,1),$  such that its covariance
function equals
\[
\E C^H_t C^H_s=t^{2H}+s^{2H}-\frac{1}{2}(|t+s|^{2H}+|t-s|^{2H}), \quad t,s\ge 0.
\]
This process was introduced in \cite{Bojdecki} in connection with the occupation
time fluctuations of branching particle systems. In the case $H=1/2$, it coincides with the standard Brownian motion: $C^{1/2}=B^{1/2}$. For $H\neq 1/2,$ $C^H$ is, in a sense, a process intermediate between the standard Brownian motion $B^{1/2}$ and the fBm $B^H$. Sub-fractional Brownian motion $C^H$ satisfies the inequalities  (see~\cite{Bojdecki} for details)
\[
|t-s|^{H}\le \|C_t^H-C_s^H\|_2 \le \sqrt{2-2^{H-1}}  |t-s|^{H}, \quad t,s\ge 0.
\]

\smallskip

\noindent{\bf 2.~Bi-fractional Brownian motion.} Let $B^{H,K}=(B_t^{H,K})_{t\ge0},$  where $H\in(0,1)$, $K\in(0,1]$ are parameters, be a zero-mean Gaussian process with covariance function
\[
\E B^{H,K}_t B^{H,K}_s= {2^{-K}}\big((t^{2H}+s^{2H})^{K}-|t-s|^{2HK}\big), \quad t,s\ge 0.
\]
This process can be considered as an extension of the fBm,
the latter being a special case when $K=1$ (see~\cite{HoudreVilla,Russo,
LeiNualart}). The process $B^{H,K}$ satisfies the following version of~\eqref{corr-ineq}:
\[
2^{-K/2}|t-s|^{HK}\le \|B_{t}^{H,K}-B_{s}^{H,K}\| \le 2^{(1-K)/2}|t-s|^{HK}, \quad t,s\ge 0.
\]

\smallskip

\noindent{\bf 3.~Fredholm and Volterra Gaussian processes.} A large class of Gaussian processes can be
obtained via a Fredholm representation:
\[
X_t = \int_0^1 {\mathcal K}(t,s) d B_s^{1/2}, \qquad t\in [0,1],
\]
with some kernel ${\mathcal K}\in L^2([0,1]^2)$ and a standard Brownian motion $B^{1/2}$.
In the case when ${\mathcal K}(t,s) = 0$ for all $s>t$, this  is called a Volterra representation.

It was proved in~\cite{sottinen} that a zero-mean Gaussian process on $[0,1],$ which is
separable in the sense  that the Hilbert space $L^2(\Omega,
\sigma(X), \P)$ is separable (note that all continuous Gaussian processes are separable),  admits a Fredholm  representation if and only
if its covariation function $R(t,s): = \E X_t X_s$ satisfies the condition
$\int_0^1 R(t,t)dt <\infty$. Then it follows from It\^o's isometry that a sufficient condition for~\eqref{corr-ineq} to hold is that
\[
C_1^2|t-s|^{2H_1}
 \le
 \int_0^1( {\mathcal K} (t,u) - {\mathcal K}(s,u))^2 du
 \le C_2^2|t-s|^{2H_2}
\quad\mbox{for all}\quad t,s\in[0,1].
\]

\smallskip

\noindent{\bf 4.~Wiener integrals\ with respect to fBms.} For a  non-random measurable function
$f:[0,1]\to {\bf R}$ and a fixed $H\in (0,1),$ consider the
process
\[
X_t := \int_0^t f(s) d B_s^H, \quad t\in[0,1].
\]
If $f(t)$ is H\"older continuous
with exponent $\alpha >1-H$ then the integral can be
defined as a pathwise Stieltjes integral. The thus defined integral will coincide with the one given by the following   definition employing fractional integrals  which is applicable in the general case.

Recall that the right-sided Riemann--Liouville integral $(\Ii^\alpha_{1-}f)(t)$
of order $\alpha\in[0,1]$ of the function $f$ is defined as follows:  $(\Ii^0_{1-}f)(t)
:= f(t),$  $t\in [0,1],$ while  for $\alpha\in(0,1]$ one puts
\[
(\Ii^\alpha_{1-}f)(t) : = \frac{1}{\Gamma(\alpha)} \int_t^1 f(s)
(s-t)^{\alpha-1} ds, \quad t\in [0,1].
\]
For $\alpha\in(-1,0),$ define $(\Ii^\alpha_{1-}f)(t)$ as the fractional
derivative of $f$ of  order $|\alpha|$:
\[
(\Ii^\alpha_{1-}f)(t) := (\mathcal{D}^{|\alpha|}_{1-}f)(t) = -\frac{d}{dt} (\Ii^{\alpha+1}_{1-}f)(t).
\]
Let $K^H$ be the linear operator that maps   $f:[0,1]\to{\bf R}$
to the function
\[
(K^H f)(t): = C_H   t^{1/2-H} (\Ii^{H-1/2}_{1-} [(\,\cdot\,)^{H-1/2} f])(t), \quad t\in[0,1],
\]
where $(\,\cdot\,)^{H-1/2}$ denotes the function $t\mapsto t^{H-1/2}$  and $C_H$ is the
normalizing constant
\[
C_H := \biggl(\frac{2H\Gamma(3/2-H)}{\Gamma(2-2H)\Gamma(H+1/2)}\biggl)^{1/2}.
\]
The fBm $B^H$ admits the representation (see e.g.\ Section~1.8 in~\cite{mishura} and~\cite{Jost})
\[
B_t^H = \int_0^1 (K^H\I_{[0,t)})(u) dB^{1/2}_u, \quad t\in [0,1],
\]
where $\I_{[0,t)}$ is the indicator function of~$[0,t)$. Now, for   $f$ such that $K^Hf \in L^2([0,1]),$  the integral of $f$ with respect to $B^H$ can be defined as
\[
X_t := \int_0^t f(u) d B_u^H = \int_0^1 (K^H[f \, \I_{[0,t)}])(u) d B^{1/2}_u, \quad t\in [0,1].
\]

Next we will give simple sufficient conditions that ensure that  relations~\eqref{corr-ineq} with $H_1=H_2=H$  hold for the integral process $X=(X_t)_{t\in [0,1]}$.

For $H\ge 1/2$, using It\^o's isometry, it is easy to see that a sufficient
condition for the left inequality in~\eqref{corr-ineq} to hold is that either $f(t) \ge c > 0$ for all $t\in[0,1]$ or $f(t) \le c < 0$ for all $t\in[0,1]$. For the right
inequality in~\eqref{corr-ineq}  it  suffices  that $|f(t)|$ is bounded on $[0,1]$.

For $H<1/2$, assume for simplicity that $f$ is absolutely continuous
on $[0,1]$ and set $h: = H-1/2$. Then the fractional integral in the definition of $K^H$ exists
and
\[
\begin{split}
(\Ii_{1-}^{h} [(\,\cdot\,)^h f\, \I_{[0,t)}])(u)
  &=
f(t)(t-u)^{h}\I_{[0,t)}(u) -\int_u^1(v^{h}f(v))'\I_{[0,t)}(v) (v-u)^{h}dv
 \\
  &=
f(t)(t-u)^{h}\I_{[0,t)}(u)
-h\int_u^1v^{h-1}f(v)\I_{[0,t)}(v) (v-u)^{h}dv
 \\
  &\hphantom{ = f(t)(t-u)^{h}\I_{[0,t)}(u)\ }-
\int_u^1v^{h}f'(v)\I_{[0,t)}(v) (v-u)^{h}dv,
\end{split}
\]
which can be obtained by integrating by parts (see e.g.\ Section~2 in~\cite{skm}).
Then a sufficient condition for the left inequality in~\eqref{corr-ineq} to hold is that, for some $c>0,$
\[
f(t)\ge c ,\quad f(t) - h^{-1}t f'(t) \ge c
 \quad\mbox{for all}\quad t\in[0,1].
\]
Indeed, in this case, using the preceding formula one obtains that
\[
(\Ii_{1-}^{h}[(\,\cdot\,)^h f\I_{[0,t)}])(u) \ge  c (\Ii_{1-}^{h}[(\,\cdot\,)^h \I_{[0,t)}])(u) \ge 0
\]
and so, by It\^o's isometry,
\[
\E |X_t-X_s|^2 \ge c^2 \int_0^1 (K^H[f \, \I_{[s,t)}])^2(u) d u =
c^2 \E(B_t^H - B_s^H)^2.
\]
Similarly, for the right inequality in~\eqref{corr-ineq} to hold it suffices  that, for some $c<\infty$,
\[
|f(t)|\le c, \quad  |f(t) - h^{-1}t f'(t)| \le c
\quad\mbox{for all}\quad  t\in[0,1].
\]
Then we have
\[
|(\Ii_{1-}^{h}[(\,\cdot\,)^h f\I_{[0,t)}])(u)|
 \le  2c (\Ii_{1-}^{h}[(\,\cdot\,)^h \I_{[0,t)}])(u)
\]
and, again by It\^o's isometry,
\[
\E |X_t-X_s|^2 \le
4c^2 \E(B_t^H - B_s^H)^2.
\]

\section{Appendix: Inequalities for Gaussian processes}
%\setcounter{section}{1}
%\renewcommand{\thesection}{\Alph{section}}
%\subsection{Inequalities for Gaussian processes}

Here we collected  a few classical bounds for the expectations of the maxima  of Gaussian processes that we needed for deriving our results.

\begin{proposition}[Sudakov--Fernique's inequality, \cite{fern,suda,vitale}]\label{lemmSF} Let  $X = (X_t)_{0\le t\le 1}$ and $Y =
(Y_t)_{0\le t\le 1}$ be two continuous Gaussian processes
such that $\E X_t = \E Y_t$ and $\|Y_{s}-Y_{t}\|_{2}\le \|X_{s}-X_{t}\|_{2}$ for all $s,t\in [0,1]$.
Then
\begin{equation*}
\E\max_{0\le t\le 1}Y_{t}\le \E\max_{0 \le t \le 1}X_{t}.
\end{equation*}
\end{proposition}

\begin{remark}
Sudakov~\cite{suda} and Fernique~\cite{fern} proved the inequality from the claim of Proposition~\ref{lemmSF}  under the assumption that the processes $X$ and $Y$ are
zero-mean. Vitale~\cite{vitale} extended the assumptions requiring only that $\E X_t = \E Y_t$.
\end{remark}

\begin{proposition}[Chatterjee's inequality, \cite{chat}]\label{chat}
Let $(X_1, \ldots, X_n)$ and $(Y_1, \ldots, Y_n)$ be Gaussian   vectors such that $\E X_i  = \E Y_i ,$  $i=1,\ldots, n$. Set
$a_{ij}: =  \|X_i- X_j\|^2_2$  and $b_{ij} := \|Y_i- Y_j\|^2_2,$ $i,j=1,\ldots, n$. Then
\[
\Bigl|\E \max_{1\le i\le n}X_i - \E \max_{1\le i\le n}Y_i\Bigr|\le
\sqrt{\max_{1\le i,j\le n}|a_{ij}-b_{ij}|\cdot \ln n}.
\]
\end{proposition}

\begin{proposition}[Sudakov's inequality, \cite{suda}]
 \label{sudakov lemma}
Let $(X_1,\ldots,X_n)$ be a zero-mean Gaussian vector. Then
\begin{equation*}
\E\max_{1\le i\le n}X_i\ge \sqrt{\frac{\log_2 n}{ 2\pi}} \min_{i\ne j}
\| X _{i}-X _{j}\|_2.
\end{equation*}
\end{proposition}

\begin{proposition}[Generic chaining bound, \cite{Talagrand, Talagrand1, talagrand}]
 \label{tala}
Let $X = (X_t)_{0\le t\le 1}$ be a continuous zero-mean Gaussian process and
$T_0\subset T_1\subset T_2\subset \cdots \subset T \subset [0,1]$ a sequence of sets
such that $|T_0| = 1$ and $|T_n|\le 2^{2^n},$ $n\ge 1.$ Then
\[
\E\max_{t\in T}X_t \le
L\max_{t\in T}\sum_{n\ge 0} 2^{n/2}  \min_{s\in T_{n}}  \|X_{t}-X_{s}\|_2, \quad L< 3.75.
\]
%where $L< 3.75$.
\end{proposition}

\begin{remark}The literature on bounds for Gaussian processes
usually does not provide explicit bounds for the constant $L$. The bound
$L<3.75$ can be obtained from the proof in~\cite{Talagrand1}. Note also  that if the set $T$   is finite, then the sum in the bound will
contain a finite number of terms because in this case $T_n = T$ for $n$
large enough (excluding the trivial case when $T_n$ ``stop growing" at some point and one  never has $T_n=T$).
\end{remark}

\noindent\textbf{Acknowledgements.}  The last author is grateful to V.I.~Piterbarg for interesting discussions of the topic of the paper.

\medskip
\noindent\textbf{Funding.}
The research was supported by the Australian Research Council under Grant DP150102758; and Russian Science Foundation under Grant~14-21-00162.

\end{document}